\documentclass[12pt,a4paper]{amsart}
\usepackage{amsmath, amsthm, amssymb,,verbatim}
\usepackage{graphicx,amssymb,amsfonts,epsfig,amsthm,a4,amsmath,url,graphicx,enumerate}
\usepackage[latin1]{inputenc}
\usepackage[usenames,dvipsnames]{color}

\newtheorem*{Ack}{Acknowledgements}

\newtheorem{thm}{Theorem}[section]
\newtheorem{cor}[thm]{Corollary}
\newtheorem{lem}[thm]{Lemma}

\newtheorem{prop}[thm]{Proposition}
\theoremstyle{definition}
\newtheorem{defn}[thm]{Definition}

\theoremstyle{remark}
\newtheorem{rem}[thm]{Remark}

\numberwithin{equation}{section}

\newcommand{\Z}{\mathbf{Z}}

\newcommand{\N}{\mathbf{N}}
\newcommand{\F}{\mathbf{F}}

\newcommand{\Q}{\mathbf{Q}}

\newcommand{\OO}{\mathcal{O}}

\newcommand{\SL}{\textnormal{SL}}

\newcommand{\kk}{\mathbf{k}}
\newcommand{\kl}{\mathbf{l}}

\title{Local-to-global rigidity of Bruhat-Tits  buildings}

\date{\today}
\begin{document}
\author[De La Salle]{Mikael De La Salle} \address{UMPA, ENS-Lyon\\ Lyon\\FRANCE}\email{mikael.de.la.salle@ens-lyon.fr}

\author[Tessera]{Romain Tessera} \address{Laboratoire de Math\'ematiques, Universit\'e Paris-Sud 11\\Orsay\\FRANCE}\email{romtessera@gmail.com}

\begin{abstract} 
A vertex-transitive graph $X$ is called local-to-global rigid if there exists $R$ such that every other graph whose balls of radius $R$ are isometric to the balls of radius $R$ in $X$ is covered by $X$. Let $d\geq 4$. We show that the $1$-skeleton of an affine Bruhat-Tits building of type $\widetilde A_{d-1}$ is local-to-global rigid if and only if the underlying field has characteristic $0$. For example the Bruhat-Tits building of $\SL(d,\F_p((t)))$ is not local-to-global rigid, while the Bruhat-Tits building of $\SL(d,\Q_p)$ is local-to-global rigid. \end{abstract}

\maketitle

A vertex-transitive graph $X$ is called local-to-global rigid (LG-rigid) if there exists $R$ such that every other graph whose balls of radius $R$ are isometric to the balls of radius $R$ in $X$ is covered by $X$. This notion was introduced by Benjaminin and Georgakopoulos and investigated in \cite{BE,G,ST15}. It follows from these works that in many cases, Cayley graphs of finitely presented groups are LG-rigid: for instance all Cayley graphs of torsion-free lattices in simple Lie groups, or Cayley graphs of torsion-free groups of polynomial growth. We also proved (\cite{ST15}) that every finitely presented group which is not a quotient of a Burnside group admits  LG-rigid Cayley graphs. On the other hand, in \cite{ST15}, we constructed many examples  of such graphs which are not LG-rigid: e.g.\ a Cayley graph of $F_2\times F_2\times \Z/2\Z$. In this article we investigate LG-rigidity for $1$-skeletons of Bruhat-Tits buildings. 

By \emph{non-archimedean local field} we will mean a locally compact discrete valuation field (not necessarily commutative). If $K$ is a non-archimedean local field and $d \geq 3$, we denote by $X_d(K)$ the Bruhat-Tits building of type $\widetilde A_{d-1}$ constructed from $K$. Our main result characterizes,  for $d \geq 4$, the fields for which $X_d(K)$ is LG-rigid.

\begin{thm}\label{thm:main} Let $K$ be a \emph{non-archimedean local field}. If $K$ has positive characteristic and $d \geq 3$, then $X_d(K)$ is not LG-rigid. By contrast, if $K$ has characteristic $0$ and $d \geq 4$, then $X_d(K)$ is LG-rigid.
\end{thm}

Let us discuss the proof of Theorem \ref{thm:main}. There is a natural locally compact Hausdorff topology on the isomorphism classes of non-archimedean local fields where two fields are $R$-close if their residue rings $\mathcal O/\pi^R \mathcal O$ are isomorphic. For example, $\Q_p[p^{1/R}]$ and $\F_p((t))$ are $R$-close. Indeed, the elements of their rings of integers have a unique representation as a formal series $\sum_{n\geq 0}^\infty a_n t^n$ with $a_n \in \{0,1,\dots,p-1\}$ and as a formal series $\sum_{n \geq 0} a_n (p^{1/R})^n$ with $a_n \in \{0,1,\dots,p-1\}$ respectively. These two rings have therefore ``the same elements'', but the operations are different (without carry for $F_p((t))$ and with carry for $\Q_p[p^{1/R}]$). However when $R$ becomes large, the difference becomes smaller and smaller as the carry is sent at distance $R$, and in particular $\Q_p[p^{1/R}]$ and $\F_p((t))$ are $R$-close. 

Similarly there is a locally compact Hausdorff topology on the isometry classes of vertex-transitive locally finite graphs where two graphs are $R$-close if they have the same balls of radius $R$. The idea of the proof of Theorem \ref{thm:main} is simple and can be summarized as follows.
\begin{enumerate}[(i)]
\item\label{item:continuity} For these topologies, the map $K \mapsto X_d(K)$ is a homeomorphism on its image.
\item\label{item:krasner} A non-archimedean local field is isolated if and only if it has characteritic $0$. 
\item\label{item:buildings_open} If $d \geq 4$, $\{X_d(K), K \;\textrm{non-archimedean local field}\}$ is open in the set of large-scale simply connected graphs.
\end{enumerate}
 The first statement in Theorem \ref{thm:main} is immediate from the continuity of $X_d$ in (\ref{item:continuity}) and from (\ref{item:krasner}). The second statement follows from (\ref{item:continuity}--\ref{item:buildings_open}) and our work \cite{ST15}.

We prove in Corollary \ref{cor:Tits} that the map $K \mapsto X_d(K)$ is continuous (actually $1$-Lipschitz for the natural distances). Since it is injective by a deep Theorem of Tits and clearly proper, (\ref{item:continuity}) follows. We could not find a direct proof of (\ref{item:continuity}); for example we could not decide whether $X_d$ is isometric. The point (\ref{item:krasner}) is classical, at least as far as commutative fields are concerned \cite{Krasner,Deligne}; its simplest illustration is, as recalled above, that $\F_p((t))$ is the limit of $\Q_p[p^{1/R}]$ as $R \to \infty$. We recall this in \S \ref{section:fields}. The meaning of (\ref{item:buildings_open}) is made precise in Corollary \ref{cor:buildings_local}; it is proved as a consequence of other deep results of Tits which give a local characterization of the graphs $X_d(K)$ among graphs with a special kind of labelling of the vertices called a \emph{geometry of type $\widetilde{A}_{d-1}$} and from our Proposition \ref{prop:geometry_local} where we show that such a labelling can be recovered locally.

The reason why we allow non-commutative fields is not to make the exposition hard to follow: even if we were only interested in commutative fields, we would have to work with non-commutative fields in the proof of the second statement. Indeed, we do not know of a direct proof showing that $\{ X_d(K), K\textrm{ commutative}\}$ is open (this is true \emph{a posteriori} because the set of commutative non-archimedean local fields is open in the space of all commutative non-archimedean local fields).

Let us state two consequences. The following is a consequence of  \cite[Corollary 1.6]{ST15} and of Theorem \ref{thm:main} (or rather of (\ref{item:continuity}) and of the convergence of $\Q_p[p^{1/R}]$ to $\F_p((t))$).
\begin{prop}\label{cocompact}
Let $N\in \N$, then for all but finitely many $R\in \N$, the building $X_d(\Q_p[p^{1/R}])$ does not admit any discrete group of isometry $\Gamma$ such that 
\begin{itemize}
\item the cardinality of the vertex set of the quotient graph  $X(\Q_p[p^{1/R}])/\Gamma$ is at most $N$;
\item for all vertex $x$ of $X_d(\Q_p[p^{1/R}])$, the stabilizer of $x$ in $\Gamma$ has cardinality at most $N$.
\end{itemize}
\end{prop}
In particular, for all but finitely many $R\in \N$, $X_d(\Q_p[p^{1/R}])$ is not a Cayley graph. For $d=3$, $X_d(\F_p((t)))$ turns out to have a group of isometries acting simply transitively on its vertex set \cite{CMSZ}. Hence $X_d(\F_p((t)))$ can be seen as a Cayley graph of this group. In particular, Theorem \ref{thm:main} yields to new examples of Cayley graphs of finitely presented groups which are not LG-rigid.

\begin{thm}\label{prop:Building}
There exists a Cayley graph $X$ of some finitely presented group, and for each $R>0$, a $2$-simply connected vertex transitive graph $Y_R$ which is $R$-locally $X$, but is not even quasi-isometric to $X$. In particular $X$ is not LG-rigid.
\end{thm}

\begin{Ack} We thank Sylvain Barr\'e, Laurent Berger, Ga\"etan Chenevier, Gabriel Dospinescu, Fr\'ed\'eric Haglund, Pierre Pansu, Mikael Pichot, Vincent Pilloni and Sandra Rozensztajn for useful discussions.
\end{Ack}
\section{Non-archimedean local fields}\label{section:fields}

\begin{defn} A non-archimedean local field is a (not necessarily commutative) field which is locally compact for a discrete valuation.
\end{defn}
%All non-archimedean local fields are isomorphic to finite extensions of either $\Q_p$ or $\F_p((t))$ for some prime number $p$ (\cite[page 20]{W74} or Lemma \ref{lem:finiteExtStruct}). If $q=p^m$ for some $m \geq 1$ we denote by $\Q_q$ the totally unramified extension of degree $m$ of $\Q_p$. 

If $K$ is a non-archimedean local field with a discrete valuation $v\colon K \to \Z \cup \{\infty\}$, we will always assume that the image of $v$ is $\Z\cup \{\infty\}$, and we will denote its ring of integers $\mathcal O=\{x \in K, v(x)\geq 0\}$ (or $\mathcal O^K$ if we need to keep track of $K$) and $\mathfrak m=\{x \in K, v(x) \geq 1\}$ the unique prime ideal in $\mathcal O$. Denote by $\pi$ a uniformizer of $K$, that is a generator of the $\mathcal O$-module $\mathfrak{m}$. For an integer $R \geq 1$ denote by $\mathcal O_R$ the ring $\mathcal O/\pi^R \mathcal O$.

A field $K$ with a discrete valuation $v$ is locally compact if and only if it is complete and the residue field $\mathcal O/ \pi \mathcal O$ is finite.

Two non-archimedean local fields are $R$-close if their residue rings $\mathcal O_R$ are isomorphic. 

\begin{thm}\label{thm:distance_nalf} The distance
\[ d(K,K') = \inf\{ e^{-R},\textrm{ $K$ and $K'$ are $R$-close}\}\]
defines a locally compact Hausdorff topology on the isomorphism classes (as topological fields) of non-archimedean local fields. 

For this topology, a set of non-archimedean local fields is relatively compact if and only if the cardinality of the residue field is bounded on this set. 

A field is isolated if and only if it has characteristic $0$.
\end{thm}

The topology is Hausdorff because a non-archimedean local field is determined as a topological field by the sequence of its residual rings%: it is isomorphic to the field of fractions over its ring of integers, which is isomorphic, as topological ring, to the inverse limit of its residual rings
. For the rest of the proof of Theorem \ref{thm:distance_nalf} we will need an explicit description of all non-archimedean local fields. We refer to \cite[Chapter 1]{W74} for the statements for which we do not provide a precise reference. If $p$ is a prime number and $q=p^f$ for some $f\geq 1$, we denote by $\Q_q$ the totally unramified extension of degree $f$ of $\Q_p$. 

Let $K$ be a \emph{commutative} non-archimedean local field with valuation $v$. Its residual field is isomorphic to $\F_{p^f}$ for a prime number $p$ called the \emph{residual characteristic of $K$} and an integer $f \geq 1$ called the absolute residual degree of $K$. The value $e=v(p) \in \N \cup \infty$ is called the absolute ramification index. If $K$ has characteristic $0$ (\emph{i.e.} $e<\infty$), then $K$ is isomorphic to a totally ramified commutative extension of $\Q_{p^f}$ of degree $e$ (hence $K$ is isomorphic to $\Q_{p^f}[X]/(P)$, where $P$ is an Eisenstein polynomial of degree $e$ with coefficients in $\OO^{\Q_{q}}$). We will use that, for every prime number $p$ and every integers $f,e \geq 1$, there are finitely many \cite[\S 3.1.6]{Robert}, and at least one (for example $\Q_{p^f}[p^{1/e}]$), commutative non-archimedean local fields with residual characteristic $p$, absolute ramification index $e$ and absolute residual degree $f$. If $K$ has positive characteristic (\emph{i.e.}\ $e=\infty$), $K$ is isomorphic to $\F_{p^f}((t))$. %The following lemma is essentially due to Krasner \cite{Krasner}.
\begin{lem}\label{lem:krasner} If $K$ is a local field of residual field $\F_q$ and absolute ramification index $e$, then $K$ and $\F_q((t))$ are $e$-close.
\end{lem}
\begin{proof} Note that $K$ is a totally ramified commutative extension of $\Q_q$ of degree $e$, Hence we can write $K = \Q_{q}[X]/(P)$ where $q=p^f$ and $P=X^e+a_{e-1}X^{e-1}+\ldots +a_0$ is an Eisenstein polynomial of degree $e$ with coefficients in $\Z_q=\mathcal O^{\Q_q}$. Recall that $P$ satisfies $a_i\in p\Z_q$ for all $0\leq i\leq k-1$, and $a_0=bp$ where $b$ is invertible in $\Z_q$. Note that $\OO^K_e = \Z_q[X]/(J)$ where $J=(X^e) + (P)$. It follows that in $\OO^K_e$ we have $$p(b+a_1X+\ldots+a_{e-1}t^{e-1})=0,$$ from which we deduce that $p=0$.
On the other hand, modulo $p$, one has $P=X^e$. So finally we deduce that
$$O_e^{K}\simeq \Z_{q}[X]/\left((X^e)+(p)\right)\simeq \left(\Z_{q}[X]/(p)\right)/(X^e)\simeq \F_{q}[X]/(X^e)\simeq \mathcal O_e^{\F_{q}((t))},$$
 and we are done.
\end{proof}

Let us now move to non-commutative fields. If $K$ is a non-archimedean local field with center $L$, the residual field $\kk$ of $K$ is an extension of the residual field $\kl$ of $L$; denote by $d$ the degree of $\kk/\kl$ ($d$ is called the residual degree of $K/L$). The Galois group Gal($\kk/\kl$) is cyclic of order $d$ with generator the Frobenius automorphism. Moreover, if $\pi$ is a uniformizer of $K$, the conjugation $x \in K \mapsto \pi^{-1} x \pi$ belongs to Gal($K/L$). Its image $\alpha$ in Gal($\kk/\kl$) does not depend on the choice of the uniformizer because $\kk$ is commutative, and corresponds to the $r$-th power of the Frobenius for some $r \in \Z/d\Z$. For convenience we shall call $r$ the {\it Hasse invariant} of $K$. The next lemma states that $r$ is a generator of $\Z/d\Z$, and that the triple $(L,d,r)$ with $d \geq 1$ and $r$ a generator of $\Z/d\Z$ determines a unique non-commutative field $K$. Observe that the case where $K$ is commutative corresponds to $d=1$, in which case $r=0$ (which is a generator of the trivial group $\Z/\Z$).

\begin{lem}(\cite{Hasse},\cite[Chapter 1 (p20--22), Chapter XII]{W74})\label{lem:ncext}
% Let $K$ be a non-archimedean field of center $L$. Let $d$ be the residual degree of $K/L$. Then $K/L$ has degree $d^2$, and $K$ contains a maximal commutative subfield $K_1$ which is totally unramified and of degree $d$ over $L$. $K$ has a uniformizer $\pi$ such that $\pi^d$ is a uniformizer of $L$, such that $\{1,\pi,\dots,\pi^{d-1}\}$ is a basis of $L$ as a vector space over $K_1$ and such that the conjugation $x \in K \mapsto \pi x \pi^{-1}$ induces on $K_1$ an automorphism which generated the (cyclic) Galois group of $K_1$ over $K$.
Let $K,L,d,r$ as above. Then $r$ is a generator of $\Z/d\Z$. Conversely, for every commutative non-archimedean local field $L$, every integer $d \geq 1$ and every generator $r$ of $\Z/d\Z$, there is a unique non-archimedean local field $K$ with center $L$, residual degree $d$ over $L$ and Hasse invariant $r$. It has degree $d^2$ over $L$ and can be described as follows. It contains a maximal commutative extension $K_1$ of $L$ of degree $d$ which is unramified. Moreover, $L$ has a uniformizer $\pi$ and $K$ has a uniformizer $x$ such that $x^d=\pi$, $(1,x,\ldots,x^{d-1})$ forms a basis of $K$ as a $K_1$-vector space, and for all $a\in K_1$, $x^{-1} a x=f^r(a)$, where $f$ is the unique automorphism of $K_1$ inducing the Frobenius automorphism of Gal($\kk_1/\kl$).
\end{lem}
%% \begin{rem}\label{rem:def_generators_relations}
%% Alternatively, one can define $K$ as follows. Let $\pi$ be a uniformozer of $L$. Let $\tilde{K}_1$ be the unique (up to isomorphism) totally unramified extension of $L$ of degree $d$. Let  $\tilde{K}$ be the quotient of the ring freely generated by $\tilde{K}_1$ and by an element $x$, by the ideal generated by the following relations: $x^d=\pi$, and $a x= x f^r(a),$ for all $a\in \tilde{K}_1$, where $f$ is the unique automorphism of $\tilde{K}_1$ inducing the Frobenius automorphism of Gal($\kk_1/\kl$). Clearly one has a morphism of $L$-algebras from $\tilde{K}$ to $K$ sending $\tilde{K}_1$ to $K_1$.  The fact this is an isomorphism follows by comparing the dimensions over $L$. In the sequel, we will not make the distinction between $\tilde{K}$ and $K$.
%% \end{rem}
We will say that a non-archimedean local field $K$ has type $(p,f,e,d,r)$ if its center $L$ has residual characteristic $p$, absolute ramification index $e$ and absolute residual degree $f$, and if the extension $K/L$ has residual degree $d$ and Hasse invariant $r$. From the preceding discussion we conclude that for every prime number $p$, every integers $f,d \geq 1$, every $e \in \N \cup \{\infty\}$ and every $r\in (\Z/d\Z)^*$, the number of fields of type $(p,f,e,d,r)$ is finite and nonzero. As we have seen, there exists a unique field of type $(p,f,\infty,d,r)$. If $q=p^f$, it can be concretely defined as the quotient of the $\F_{q}((t))$-algebra freely generated by $\F_{q^{d}}((t))$ and by an element $x$, by the ideal generated by the following relations: $x^d=t$, and $a x= x f(a),$ for all $a\in \F_{q^{d}}((t))$, where $f$ is the automorphism of $\F_{q^{d}}((t))$ uniquely defined by $f(t)=t$, and $f(z)=z^{q^{r}}$ for all $z\in \F_{q^d}$. 

\begin{lem}\label{lem:BallLemma}
Let $K$ (resp.\ $F$) be respectively a field of type $(p,f,e,d,r)$ (resp.\ the field of type $(p,f,\infty,d,r)$). Then $K$ and $F$ are $ed$-close: i.e.\ the residue rings $\mathcal O^K_{ed}$ and  $\mathcal O^{F}_{ed}$ are isomorphic.
\end{lem}
\begin{proof}
With the notation of Lemma \ref{lem:ncext}, $K$ is isomorphic to the quotient $\widetilde K$ of the ring freely generated by $K_1$ and $x$, by the ideal generated by the relations $x^d=\pi$, and $a x= x f^r(a),$ for all $a\in K_1$, where $f$ induces the Frobenius on the residue field. Indeed, one clearly has a morphism of $L$-algebras from $\widetilde{K}$ to $K$ that is the identity on $K_1$ and on $x$. The fact this is an isomorphism follows by comparing the dimensions over $L$. We deduce that the residue ring $\OO^K_{ed}=\OO^K/(x^{ed})=\mathcal O^K/(\pi^{e})$ is isomorphic to the quotient of the ring freely generated by $\mathcal O^{K_1}/(\pi^{e})$ and $x$, by the ideal generated by the relations $x^d=\pi$, and $a x= x \widetilde f^r(a)$, where $\widetilde f$ is the unique (by Hensel's Lemma) automorphism of the ring $\mathcal O^{K_1}/(\pi^{e})$ which induces the Frobenius on the residue field. The same description of course applies to $\mathcal O^{F}_{ed}$, and we conclude by Lemma \ref{lem:krasner} because $K_1$ has absolute residual degree $fd$ and absolute ramification index $e$.
\end{proof}

We can now prove Theorem \ref{thm:distance_nalf}. To do so we show that the balls of radius $\frac 1 2$ are compact (here $\frac 1 2$ could be any number in $[\frac 1 e, 1)$), and that their only accumulation points are the fields of positive characteristic. Since two fields are at distance less than $\frac 1 2$ if and only if they have the same residue field, it amounts to investigating, for every finite field $\F_q$, the set of fields having $\F_q$ as residue field. This sets contains exactly the fields of type $(p,f,e,d,r)$ for $q=p^{fd}$ and $k \in \N \cup \{\infty\}$. This determines $p$ and forces $f,d$ to take only finitely many values. Therefore (since there are finitely many fields of each type) a sequence of such fields either has a stationnary subsequence, or a subsequence of type $(p,f,e_n,d,r)$ for a sequence $e_n \to \infty$, which converges to the field of type $(p,f,\infty,d,r)$ by Lemma \ref{lem:BallLemma}. This shows that the set of fields with residue field $\F_q$ is compact, and that the fields with characteristic $0$ are isolated. Conversely, every nonarchimedean local field $K$ of characteristic $p>0$ is the field of type $(p,f,\infty,d,r)$ for some $f,d,r$. As we discussed there is a sequence of fields of type $(p,f,n,d,r)$, and it converges as $n \to \infty$ to $K$ by Lemma \ref{lem:BallLemma}.

\section{Buildings}

\subsection{Graphs} In this paper ``a graph" means a connected, locally finite, simplicial graph withouth multiple edges and loops. It is called vertex-transitive if its isometry group acts transitively on the set of vertices. A graph $Y$ is $R$-locally $X$ if every ball of radius $R$ around a vertex in $Y$ is isometric to a ball of radius $R$ around a vertex in $X$. This defines a locally compact Hausdorff topology on the isomorphism classes of transitive graphs, for example for the distance 
\[\inf\{ e^{-R},\textrm{ $Y$ is $R$-locally $X$}\}.\] A set of vertex-transitive graphs if relatively compact if and only if the degree is bounded on this set.

\subsection{Classical buildings of type $\widetilde A_{d-1}$} 
Let $d\geq 2$. Let us recall the description of the building of $GL(d,K)$ (the building $\widetilde A_{d-1}(K,v)$) associated to a non-archimedean local field $K$ with discrete valuation $v$, see \cite[Chapter 9]{Ronan} for details.

An $\mathcal O$-lattice in $K^d$ is a finitely generated $\mathcal O$-submodule which generates $K^d$ as a $K$-vector space. Such a module is free of rank $d$, \emph{i.e.} of the form $\mathcal O v_1 + \ldots + \mathcal O v_d$ for a basis $(v_1,\ldots,v_d)$ of $K^d$. By the invariance property $a \mathcal O = \mathcal O a = \pi^k \mathcal O$ for any $a \in K^*$ and $k \in \Z$ with $v(a) = k$, we see that if $L$ is an $\mathcal O$-lattice and $a \in K^*$, $aL$ is also a lattice, so that it makes sense to talk about lattices modulo homothety.

The building $\widetilde A_{d-1}(K,v)$ is a simplicial complex of dimension $d-1$. Its $1$-squeleton, that we denote by $X_d(F)$ (or $X$ for short if there is no ambiguity) is described as follows. The vertices of $X$ are the $\mathcal O$-lattices in $K^d$ modulo homothety. There is an edge between two different vertices $x$ and $y$ if there are representatives $L_1$ and $L_2$ of $x$ and $y$ such that $\pi L_1 \subset L_2 \subset L_1$. This is the vertex transitive graph $X_d(F)$ we are interested in. 

\subsection{Continuity of $K \mapsto X_d(K)$} 
A lattice modulo homothety $x$ has a unique representative, denoted by $L(x)$, contained in $\mathcal O^d$ but not in $\mathfrak{m}^d$. There is an edge between two different vertices $x$ and $y$ if and only if $\pi L(x) \subset L(y) \subset L(x)$ or $\pi L(y)  \subset L(x) \subset L(y)$.

The following Lemma expresses that the ball of radius $R$ around $\mathcal O^d$ in $X$ is entirely described in terms of the ring $\mathcal O_R$.

\begin{lem}\label{lem:building_from_residue_ring} A lattice modulo homothety $x$ belongs to the ball of radius $R$ around $o$ if and only if $\pi^{R} \mathcal O^d \subset L(x)$.

Moreover the map $\overline L \colon x \mapsto L(x) \mod \pi^R \mathcal O^d$ is a bijection between the ball of radius $R$ around $\mathcal O^3$ in $X$ and the $\mathcal O_R$-submodules of $(\mathcal O_R)^d$ not contained in $(\pi \mathcal O_R)^d$.

Lastly two different vertices $x$ and $y$ in the ball of radius $R$ around $\mathcal O^d$ in $X$ are adjacent if and only if $\pi \overline{L}(x) \subset \overline{L}(y) \subset \overline{L}(x)$ or $\pi \overline{L}(y)  \subset \overline{L}(x) \subset\overline{L}(y)$.
\end{lem}
\begin{proof}
It is immediate that $\pi^{R} \mathcal O^d \subset L(x)$ if $d(x,o) \leq R$. The converse follows by applying, for any lattice $L \subset \mathcal O$, the invariant factor decomposition over $\mathcal O$-modules (\cite{T37}) to $\mathcal O^d/L$, which provides a basis $v_1,\dots,v_d$ for the $\mathcal O$-module $\mathcal O^d$ and integers $n_1\leq \dots \leq n_d$ such that $\pi^{n_1} v_1,\dots,\pi^{n_d} v_d$ is a basis for $L$. For the lattice $L(x)$, we have $n_1=0$, and if $\pi^R \mathcal O^d \subset L(x)$, we have $n_d \leq R$. If $x_k$ is the equivalence class of the lattice $\oplus_{1 \leq i \leq d} \mathcal O \pi^{\min(k,n_i)} v_i$
then $x_k$ and $x_{k+1}$ are adjacent in $X_d$, $x_0=o$ and $x_{R}=x$, which shows that $d(x,y) \leq R$.

Since $L \mapsto L \mod \pi^R \mathcal O^d$ is a bijection between the lattices $L$ such that $\pi^R \mathcal O^d \subset L \subset \mathcal O^d$ and the $\mathcal O_R$-submodules of $\mathcal O_R^d$, the second statement is immediate from the first.

The last statement is easy.
\end{proof}
We immediately deduce that if two local fields $K,K'$ are $R$-close, then the graphs $X_d(K)$ is $R$-locally $X_d(K')$.
\begin{cor}\label{cor:Tits} The ball of radius $R$ in $X_d(K)$ does only depend (up to isometry) on the ring $\mathcal O_R$. 
\end{cor}
\begin{proof} By Lemma \ref{lem:building_from_residue_ring} the ball only depends on the pair $\pi \mathcal O_R \subset \mathcal O_R$. But $\pi \mathcal O_R$ is determined by  $\mathcal O_R$ as its unique maximal ideal.
\end{proof}

\subsection{$\{X_d(K)\}$ is open} We start by recalling some material from \cite{Tits}.

For an integer $m \geq 2$, a generalized $m$-gon is a connected bipartite graph of diameter $m$ and girth $2m$, in which every vertex has degree at least $2$.

A Coxeter diagram over $I$ is a function $M \colon I \times I \to \N \cup \{\infty\}$ such that for all $i,j \in I$, $M(i,i)=1$ and $M(i,j)=M(j,i) \geq 2$ if $i \neq j$. A symmetry of a Coxeter diagram $M$ is a permutation $\sigma$ of $I$ satisfying $M(i,j) = M(\sigma(i),\sigma(j))$ for all $i,j \in I$.

If $I$ is a set, a geometry over $I$ is a pair $(X,\tau)$ where $X$ is a graph and $\tau \colon X \to I$ is a coloring of the vertices of $X$ by labels in $I$ satisfying  that every pair of adjacent vertices have a different label. In a geometry, a complete subgraph is called a flag, and its type is the subset of $I$ defined as its image by $\tau$. The residue of a flag $Z$ of type $J \subset I$ is the geometry over $I \setminus J$ given by $(Y,\tau\left|_{Y}\right.)$ where $Y$ is the set of vertices in $X \setminus Z$ adjacent to $Z$, with the same edges as in $X$. By convention the residue of the empty flag is $(X,\tau)$.

If $M$ is a Coxeter diagram over a set $I$, a geometry of type $M$ is a geometry $(X,\tau)$ over $I$ where for any subset $J \subset I$, the residue of any flag of type $J$ is (1) nonempty if $|I \setminus J| \geq 1$, (2) connected if $|I \setminus J| \geq 2$, (3) a generalized $M(i,j)$-gon if $J= I \setminus \{i,j\}$ for some $i\neq j \in I$. We say that a graph admits a geometry of type $M$ if there exists $\tau \colon X \to I$ such that $(X,\tau)$ is a geometry of type $M$.

The example important for us is the Coxeter diagram $\widetilde A_{d-1}$ over $\Z/d\Z$ given by where $\widetilde A_{d-1}(i,i)=1$, $\widetilde A_{d-1}(i,j)=3$ if $i-j \in \{-1,1\}$ and $\widetilde A_{d-1}(i,j)=2$ otherwise. Then $X_d(K)$ admits a geometry of type $\widetilde A_{d-1}$ (see \cite{Ronan}). It is characterized by the following properties. The origin $o$ is labelled by $\tau(o)=0$, and if $x$ and $y$ are two adjacent vertices with representatives $L_1$ and $L_2$ such that $\pi L_1 \subset L_2 \subset L_1$, then $L_1/\pi L_1$ is a vector space of dimension $d$ over the finite field $\mathcal O/\mathfrak m$ and the dimension (modulo $d$) of the image of $L_2$ inside it is equal to $\tau(y) - \tau(x)$.

A particular case of a theorem of Tits \cite[Theorem 1.3]{Tits} characterizes the buildings of type $\widetilde A_{d-1}$ as the simply connected geometries of type $\widetilde A_{d-1}$. This motivates the following result, which shows that for a large scale simply connected graph $Y$, admitting a geometry of type $M$ is a local property.
\begin{prop}\label{prop:geometry_local} Let $M$ be a Coxeter diagram over a finite set $I$. Let $X$ be a $3$-simply connected graph. Assume that every ball of radius $3$ in $X$ is isomorphic to a ball of radius $3$ in a graph admitting a geometry of type $M$. Then $X$ admits a geometry of type $M$.
%Assume that for every $x \in X$ there is a coloring $\tau$ of the vertices at distance at most $2$ from $x$ such that properties (1) (2) and (3) hold for every flag containing a vertex at distance at most $1$ from $x$. Then $X$ admits a geometry of type $M$.
\end{prop}
\begin{rem} More generally if $X$ is $k$-simply connected, then $X$ admits a geometry of type $M$ if every ball of radius $\lfloor \frac{k+3}{2}\rfloor$ in $X$ is isometric to a ball of the same radius in a graph admitting geometry of type $M$.
\end{rem}

If $X$ is a graph and $x \in X$, let $V(x)$ be the graph with vertex set $\{x' \in X, d(x,x') \leq 1\}$ and same edges as in $X$. A \emph{germ of geometry of type $M$} at $x$ is a coloring $\tau \colon V(x) \to I$ such that $(V(x),\tau)$ is a geometry over $I$ and such that the conditions (1) (2) (3) hold for every flag in $V(x)$ containing $x$. Denote by $G(x)$ the set of all germs of geometry of type $\widetilde A_{d-1}$ at $x$. Observe that for a connected graph $X$ and a map $\tau \colon X\to I$, $(X,\tau)$ is a geometry of type $I$ if and only if the restriction of $\tau$ to $V(x)$ if a germ of geometry of type $M$ for every $x \in X$. It is through this observation that we will construct a suitable labelling of a graph satisfying the local properties of Proposition \ref{prop:geometry_local}.

We start by a lemma which implies that a germ of geometry of type $M$ at $x$ is characterized by its restriction to any flag of type $I$.
\begin{lem}\label{lem:description_of_germs} Let $X$ be a graph and $x \in X$. If $\tau \in G(x)$, $G(x)$ consists of all maps of the form $\sigma \circ \tau$ for a symmetry $\sigma$ of $M$. 
\end{lem}
\begin{proof}
It is clear that $\sigma \circ \tau \in G(x)$ if $\sigma$ is a symmetry of $M$. 

To see the converse take $\tau' \in G(x)$. Let $Z$ be a flag of type $I$ containing $x$ (such a flag exists by (1)). Then for every $z_1 \neq z_2 \in Z \setminus \{x\}$, by looking at the residue of $Z \setminus \{z_1,z_2\}$ we obtain from condition (3) that $M(\tau(z_1),\tau(z_2))=M(\tau'(z_1),\tau'(z_2))$. This implies that there exists a unique symmetry $\sigma_Z$ of $M$ such that $\tau'(z) = \sigma_Z \circ \tau(z)$ for all $z \in Z$. 
%% By the assumptions (1) and (2) and by downwards induction on the cardinality of $Z \cap Z'$, we see that $\sigma_Z = \sigma_{Z'}$ for every pair of flags $Z$ and $Z'$ of cardinality $d$ containing $x$. 
We claim that $\sigma_Z = \sigma_{Z'}$ for every pair of flags $Z$ and $Z'$ of cardinality $d$ containing $x$. The claim is proved by downwards induction on the cardinality of $Z\cap Z'$. The case when $|Z \cap Z'| = |I|$ or $|Z \cap Z'| = |I|-1$ is obvious. Assuming that the claim is valid when $|Z \cap Z'|=k \leq |I|-1$, let us prove the claim when $|Z\cap Z'| = k-1$. Pick $z \in Z\setminus Z'$ and $z' \in Z'\setminus Z$. By (2) there is a path $z_0,\dots,z_n$ contained in the residue of $Z \cap Z'$ such that $z_0=z$ and $z_n = z'$. By (1) for each $i=0,\dots,n-1$ there is a flag $Z_i$ of type $I$ containing $(Z\cap Z') \cup \{z_i,z_{i+1}\}$, and by induction hypothesis applied to $Z_i$ and $Z_{i+1}$ we have $\sigma_{Z_i} = \sigma_{Z_{i+1}}$ for each $0\leq i \leq n-1$. This proves that $\sigma_Z = \sigma_{Z'}$. Hence $\sigma_Z$ does not depend on $Z$, which proves the lemma.
\end{proof}
From it we deduce the following
\begin{lem}\label{lem:extension_of_geometry} Let $X$ be a graph admitting a geometry of type $M$ and $x$ a vertex in $X$. For every germ $\tau$ of geometry of type $M$ at $x$, there exists $\widetilde \tau \colon X \to I$ which extends $\tau$ and such that $(X,\widetilde \tau)$ is a geometry of type $M$.
\end{lem}
\begin{proof} 
Let $\tau_0 \colon X \to I$ such that $(X,\tau_0)$ is a geometry of type $I$. By Lemma \ref{lem:description_of_germs}, there is a symmetry $\sigma$ of the diagram $M$ such that $\tau$ coincides with $\sigma \circ \tau_0$ on $V(x)$. Then $\widetilde \tau = \sigma \circ \tau_0$ extends $\tau$ and satisfies that $(X,\widetilde \tau)$ is a geometry of type $M$. This shows the existence. 
\end{proof}

With these two lemmas we can proceed to the proof of Proposition \ref{prop:geometry_local}. The argument is the same as for the proof of \cite[Theorem C]{ST15}. Since every ball of radius $3$ in $X$ is isometric to a ball in a graph admitting a geometry of type $M$, $G(x)$ is nonempty for every $x \in X$, and it follows from Lemma \ref{lem:extension_of_geometry} that if $x,x'$ are two neighbors in $X$, then for every $\tau \in G(x)$, there is $\tau' \in G(x')$ which coincides with $\tau$ on $V(x) \cap V(x')$. It is unique because by Lemma \ref{lem:description_of_germs}, $\tau'$ is determined by its value on any complete subgraph with $d$ vertices containing $x'$ (and there is such a subgraph containing both $x$ and $x'$ by (1)). This defines a bijection that we denote $F_{x,x'} \colon G(x) \to G(x')$. 

For every path $\gamma=(x_0,\dots,x_n)$ of adjacent vertices, we can define a bijection $F_\gamma \colon G(x_0) \to G(x_n)$ by composing the bijections $F_{x_i,x_{i+1}}$ along $\gamma$. We claim that $F_\gamma$ only depends on the endpoints $x_0$ and $x_n$. Since $X$ is $3$-simply connected, we only have to check that $F_\gamma$ is the identity of $A(x_0)$ if $\gamma$ is a path of length $n \leq 3$ with $x_0=x_n$. This property clearly holds if $X$ admits a geometry of type $M$, and hence also in $X$ because $\cup_{k \leq n} V(x_i)$ (and all its edges) is contained in the ball of radius $3$ around $x_0$, which is isometric to a ball of radius $3$ in a graph admitting a geometry of type $M$.

It remains to fix a vertex $x_0 \in X$ and $\tau_0 \in G(x_0)$. For every other vertex $x$, let $\tau_x \in G(x)$ be the common value of $F_{\gamma}(\tau_0)$ for all paths $\gamma$ from $x_0$ to $x$. Define $\tau(x) = \tau_x(x)$. Since for adjacent edges $x,x'$, $\tau_{x'} = F_{x,x'}(\tau_x)$ coincides with $\tau_x$ on $V(x) \cap V(x')$, we see that $\tau$ coincides with $\tau_x$ on $V(x)$. In particular the restriction of $\tau$ to $V(x)$ belongs to $G(x)$ for all $x \in X$. This means that $(X,\tau)$ is a geometry of type $\widetilde A_{d-1}$. This concludes the proof of Proposition \ref{prop:geometry_local}.

We deduce the following 
\begin{cor}\label{cor:buildings_local} Let $d \geq 4$. Let $Y$ be a $3$-simply connected graph which is $3$-locally $X_d(K)$ for some non-archimedean local field $K$. Then $Y$ is isometric to $X_d(F)$ for a (unique up to isomorphism) non-archimedean local field $F$.
\end{cor}
\begin{proof} From the discussion before Proposition \ref{prop:geometry_local}, $Y$ is isometric to the $1$-skeleton of a locally finite building of type $\widetilde A_{d-1}$ (and this holds whenever $d \geq 3$). By a theorem of Tits, this forces $Y$ to be isometric to $X_d(F)$ for some $F$ if $d\geq 4$. Moreover, $X_d(F)$ determines the projective space $PG(d-1,F)$ up to collineation, which determines $F$ up to isomorphism by the fundamental theorem of projective geometry. See \cite[page 137]{Ronan}, or \cite[Corollaire 15]{Tits86} and \cite[Theorem 6.3]{Tits74}.
\end{proof}

\section{Proof of Theorem \ref{thm:main}}

The map $K \mapsto X_d(K)$ is continuous by  Corollary \ref{cor:Tits}. It is injective by the theorems of Tits recalled in the proof of Corollary \ref{cor:buildings_local}. Let us show that it is proper, that is that if the cardinality of the residue field $K_n$ goes to $\infty$, then the degree in $X_d(K_n)$ also. But this holds because (Lemma \ref{lem:building_from_residue_ring}) the degree in $X_d(K)$ is equal to the number of linear subspaces of dimension $\neq 0,d$ in\footnote{Precisely, if the residue field of $K$ is $\F_q$, the degree equals $\Pi_{i=1}^d(q^i-1)/(q-1)=\Pi_{i=1}^d (\sum_{v=0}^{i-1}q^v)$, which is a strictly increasing function of $q$.} $(\mathcal O/\pi \mathcal O)^d$. This implies (\ref{item:continuity}): $K \mapsto X_d(K)$ is a homeomorphism on its image.

The first part of Theorem \ref{thm:main} follows from Theorem \ref{thm:distance_nalf} and (\ref{item:continuity}).

Let us now prove the second half of Theorem \ref{thm:main}. Let $d \geq 4$ and $K$ be a field of characteristic $0$. By Theorem  \ref{thm:distance_nalf} and (\ref{item:continuity}), there exists $R>0$ such that if $K'$ is another non-archimedean local field such that $X_d(K')$ is $R$-locally $K$, then $K'$ is isomorphic to $K$, and in particular $X_d(K')$ is isometric to $X_d(K)$. By Corollary \ref{cor:buildings_local} this implies that if $Y$ is a $3$-simply connected graph which is $\max(3,R)$-locally $X_d(K)$, it is isometric to $X_d(K)$. By \cite[Proposition 1.5]{ST15} $X_d(K)$ is LG-rigid.

\end{document}